\documentclass[12pt, a4paper, english]{article}		
\usepackage[T1]{fontenc}
\usepackage{babel}		
\usepackage{fancyhdr}		
\usepackage[utf8]{inputenc}
\usepackage{amsmath}  
\usepackage{amsthm}		
\usepackage{amsfonts}		
\usepackage{amssymb} 
\usepackage{enumerate} 
\usepackage{graphicx} 
\usepackage[all]{xy}
\usepackage{url}
\newtheorem{teorema}{Theorem}
\newtheorem{lema}[teorema]{Lemma}
\newtheorem{proposicio}[teorema]{Proposition}
\newtheorem{corolari}[teorema]{Corollary}
\newtheorem{athm}{Theorem}

\theoremstyle{definition}

\DeclareMathOperator{\Aut}{Aut}
\DeclareMathOperator{\Ann}{Ann}

\DeclareMathOperator{\Ker}{Ker}

\DeclareMathOperator{\Soc}{Soc}

\DeclareMathOperator{\Ret}{Ret}
\DeclareMathOperator{\id}{id}

\title{On left braces in which every subbrace is an ideal II}

\author{A. Ballester-Bolinches\thanks{Departament de Matem{\`a}tiques, Universitat de Val{\`e}ncia; Dr.\ Moliner, 50; 46100 Burjassot, Val\`encia, Spain; \texttt{Adolfo.Ballester@uv.es}, ORCID 0000-0002-2051-9075; \texttt{Ramon.Esteban@uv.es}, ORCID 0000-0002-2321-8139; \texttt{Vicent.Perez-Calabuig@uv.es}, ORCID 0000-0003-4101-8656} \and R. Esteban-Romero\addtocounter{footnote}{-1}\footnotemark \and L. A. Kurdachenko\thanks{Department of Algebra and Geometry, Oles Honchar Dnipro National University, Dnipro 49010, Ukraine; \texttt{lkurdachenko@gmail.com}; ORCID 0000-0002-6368-7319}\ \thanks{Part of the research of this author has been carried out in the Departament de Matem{\`a}tiques, Universitat de Val{\`e}ncia; Dr.\ Moliner, 50; 46100 Burjassot, Val\`encia, Spain} \and V. P{\'e}rez-Calabuig\addtocounter{footnote}{-3}\footnotemark}

\date{}
\begin{document}
\maketitle

\begin{abstract}
The aim of this paper is to take the study of Dedekind braces, that is, left braces for which every subbrace is an ideal, started in a previous paper, further. Dedekind braces $A$ whose additive group is non-periodic are analysed. We prove sufficient conditions for $A$ to be abelian: it is enough that every element is $2$-nilpotent for the star operation; 
and, if $A$ is hypermultipermutational, it suffices that the additive group of the socle is torsion-free. Both conditions can be translated in terms of set-theoretical solutions of the Yang-Baxter equation. In addition, we prove a structural theorem for the case of $A$ to be a multipermutational brace of level~$2$.
\end{abstract}

\emph{Mathematics Subject Classification (2020):} 16T25, 
  16N40, 
  81R50 
\emph{Keywords:} left braces, Yang-Baxter equation, Dedekind left braces, non-periodic groups.

\section{Introduction}
\label{sec:intro}
\emph{Left braces} introduced by Rump in~\cite{Rump07} stand out as a keystone to cope the classification problem of solutions of the Yang-Baxter equation (YBE), an important equation with a wide variety of applications in both realms of Physics and Mathematics. 

A left brace is a triple $(A,+,\cdot)$ such that $(A,+)$ and $(A,\cdot)$ are group structures, $(A,+)$  abelian, satisfying
\[ a\cdot (b+c) = a\cdot b + a \cdot c - a, \ \text{for every $a,b,c\in A$.}\]
The origin of left braces comes from a generalisation of Jacobson radical rings in terms of the so-called \emph{star operation}
\[ \ast \colon (a,b)\in A \times A \longmapsto a \ast b = a\cdot b - a - b \in A.\]
In particular, a left brace is called \emph{abelian} if both operations coincide, or equivalently, $a\ast b = 0$ for every $a,b\in A$. 

One of the natural lines of attacking the classification problem of solutions is to analyse how the algebraic structure of left braces determines properties of solutions (see~\cite{Rump07}). \emph{Indecomposable} and \emph{multipermutational solutions}  stick out as particularly studied classes of solutions by means of algebraic properties of left braces (see ~\cite{CedoJespersOkninski10, CedoJespersOkninski14, CedoOkninski23, Gateva-IvanovaCameron12, Gateva-Ivanova18-advmath, JedlickaPilitowska23, Smoktunowicz18-tams}). In the former case, the so-called \emph{twist solutions} are in the antipodes of indecomposability and play a key role in the description of classes of solutions (see~\cite{BachillerCedoJespersOkninski19, BallesterEstebanJimenezPerezC24-solubleskewbraces}). Twist solutions can be described in terms of abelianity of left braces, and therefore, conditions ensuring the abelianity of a left brace are very welcomed in this context. On the other hand, the multipermutational level of a solution measures the degree of retractability of a solution (see \cite{EtingofSchedlerSoloviev99}). The study of solutions having finite multipermutational level carries a lot of weight in the classification problem. Their importance relies upon the fact that almost every solution is multipermutational (see~\cite{JespersVanAntwerpenVendramin23}, for example). These solutions are characterised by algebraic properties on the algebraic structures associated, such as left orderability and nilpotency. In this light, one of the main approaches to cope with the problem of classifying multipermutational solutions is to address first the subclass of those one with multipermutational level~$2$. To deal with this strategy, the \emph{socle series} of a left brace (see Section~\ref{sec:prelim}) is of great significance. 

In~\cite{BallesterEstebanKurdachenkoPerezC-dedekind-leftbraces}, we define \emph{Dedekind left braces} as those left braces satisfying that every subbrace is an ideal (see Section~\ref{sec:prelim}). The main results in~\cite{BallesterEstebanKurdachenkoPerezC-dedekind-leftbraces} provide complete descriptions of Dedekind left braces whose additive group is periodic. These structural theorems reveal that such Dedekind left braces provide interesting families of multipermutational solutions of level~$2$.  Bearing in mind that there are also infinite left brace structures associated with solutions, the study of Dedekind left braces undoubtedly calls to explore the non-periodic case, and this is the main aim of this paper. 

Section~\ref{sec:main} contains our main theorems. Theorem~\ref{ateo:2.11nou} provides a sufficient condition for the abelianity of a Dedekind left brace whose additive group is non-periodic: it is sufficient to look at the diagonal of the composition table of the star operation. This can be transcribed in terms of nilpotency elements (see Section~\ref{sec:prelim}), improving a result in~\cite{BallesterFerraraPerezCTrombetti23-arXiv-note}. This result has a strong implication in the classification problem of solutions: if the left brace structure of a solution is Dedekind, we obtain an easy to check sufficient condition for a solution to be a twist solution (see Corollary~\ref{cor:twistsolutions}). 

Theorems~\ref{ateo:Ahypersoc_torsionfreee->Aabelian} and~\ref{ateo:2.17nou} deal with the multipermutationality of Dedekind left braces. Theorem~\ref{ateo:Ahypersoc_torsionfreee->Aabelian} shows that Dedekind left braces whose socle is additively torsion-free are abelian. It implies that multipermutational solutions whose left brace structure is Dedekind are necessarily twist solutions, and therefore, of level $1$ (see Corollary~\ref{cor:mult+dedekind->twist}). Theorem~\ref{ateo:2.17nou}  is focused on the interesting case of left braces with multipermutational level~$2$. For those left braces, we prove a structural theorem which completely describes Dedekind left braces whose additive group is non-periodic.

\section{Preliminaries}
\label{sec:prelim}

In this section, we fix notation and present known notions and results on left braces.

Throughout all the article, we refer to left braces $(A,+,\cdot)$ simply as $A$, and we use juxtaposition for products on left braces. The group structures on a left brace have a common identity which we denote by~$0$. We follow the following hierarchy of operations: $\ast$-products come first than products, which come first than sums on left braces. Given $a\in A$, we refer to the order of $a$ in $(A,+)$ as the \emph{additive order} (finite or infinite) of $a$.

The so-called $\lambda$-action relates both group structures in a left brace $A$. It is denoted by the homomorphism $\lambda \colon a \in (A,\cdot) \mapsto \lambda_a \in (A,+)$, such that  $\lambda_a(b) = -a + ab$, for every $a,b\in A$. Therefore,
\[ ab = a + \lambda_a(b), \quad a+b = a\lambda_{a^{-1}}(b), \quad a \ast b = \lambda_a(b) - b, \quad \forall\, a,b\in A\]
The following equalities are also easily verified and used all over the article:
\begin{eqnarray}
\label{ast_distesq} a \ast (b+c) & = & a \ast b + a \ast c\\
\label{ast_prod} (ab)\ast c & = & a \ast (b \ast c) + b\ast c + a \ast c
\end{eqnarray} 

A non-empty subset $S$ of a left brace $A$ is called a \emph{subbrace} if $(S,+)$ and $(S,\cdot)$ are subgroups of $(A,+)$ and $(A,\cdot)$, respectively. We write $S \leq A$. A subbrace $I \leq A$ is an \emph{ideal} of $A$ if $A \ast I, I \ast A \subseteq I$. Ideals of left braces are $\lambda$-invariant and allow to consider a quotient left brace structure $(A/I, +, \cdot)$, as $(I,+)\unlhd (A,+)$, $(I,\cdot)\unlhd (I, \cdot)$, and  $a+I = aI$ for every $a\in A$. A \emph{subbrace generated} by a subset $S$ of a left brace $A$, $\langle S \rangle$, is the intersection of all subbraces containing $S$. We use $\langle S\rangle_+$ and $\langle S \rangle_{\boldsymbol{\cdot}}$ to respectively denote generated subgroups in the additive and the multiplicative group of $A$. We just write $\langle a\rangle$ for generated substructures by singletons $\{a\}$ with $a\in A$. If $X$ and $Y$ are subsets of $A$, $X \ast Y = \langle x \ast y\mid x \in X, \, y \in Y\rangle_+$. 

\begin{lema}[\cite{BallesterEstebanKurdachenkoPerezC-dedekind-leftbraces}]
\label{lema:a·a=0ciclic}
Let   $a \in A$ such that $a\ast a = 0$. Then, $\langle a \rangle = \langle a \rangle_+ = \langle a \rangle_{\boldsymbol{\cdot}}$ is an abelian subbrace of $A$.  
\end{lema}

\emph{Homomorphisms} between left braces are mappings conserving sums and products, and isomorphic left braces are provided by isomorphisms, i.e. bijective homomorphisms. Given two braces $A$ and $B$, the \emph{direct product} of left braces $A \oplus B$ is given by each direct products of sums and products of $A$ and~$B$. Since both group structures coincide in an abelian left brace, we can refer to them as \emph{isomorphic} to the common group structure.

\bigskip

Nilpotency notions of left braces naturally appear by iterating star products. Following~\cite{Rump07}, the \emph{left and right series} of a left brace $A$ are respectively defined as 
\begin{align*}
& A^1 = A  \geq A^2 = A \ast A \geq \ldots \geq A^{n+1} = A \ast A^n \geq \ldots\\
& A^{(1)} = A \geq A^{(2)} = A \ast A \geq \ldots \geq A^{(n+1)} = A^{(n)}\ast A \geq \ldots 
\end{align*}
$A$ is said to be \emph{left (right) nilpotent} if the left (right) series reaches $0$, and $A$ is called \emph{centrally nilpotent} if $A$ is both left and right nilpotent (see  \cite{BallesterEstebanFerraraPerezCTrombetti-arXiv-cent-nilp,BonattoJedlicka23, CedoSmoktunowiczVendramin19} for further information). 

Engel conditions in radical rings are also translated to left braces. We call a left brace $A$ $m$-\emph{right nil} if for every $a\in A$,
\[ \begin{array}{c}
\underbrace{(\ldots((a\ast a)\ast \ldots )\ast a)}_{m\;\textnormal{times}}=0;
\end{array}
\]
similarly, $A$ is $m$-\emph{left nil} if for every $a\in A$, \[
\begin{array}{c}
\underbrace{(a\ast(\ldots \ast(a\ast a))\ldots )}_{m\;\textnormal{times}}=0 
\end{array}
\] 
It follows that left nil left braces are left nilpotent (see~\cite{Smoktunowicz18-jalg}) but right nil left braces are not right nilpotent in general (see~\cite{BallesterFerraraPerezCTrombetti23-arXiv-note}). However, the following result holds.
\begin{proposicio}[\cite{BallesterFerraraPerezCTrombetti23-arXiv-note}]
\label{prop:2-nil->centnil}
Let $A$ be a finite left brace such that $a \ast a = 0$ for every $a\in A$. Then, $A$ is centrally nilpotent.
\end{proposicio}

Right and central nilpotency play a prominent role in the study of multipermutational solutions of the YBE. In this context, the following substructures stand out: given a subset $S$ of a left brace $A$,
\begin{align*}
\Ann_A^l(S) & = \{a \in A \mid a \ast x = 0,\, \text{for all $x\in S$}\} = \\
& =\{a \in A \mid ax = a+x,\,\text{for all $x\in S$}\};\\
\Ann_A(S) &  =\{a \in A \mid ax = a+x = xa,\,\text{for all $x\in S$} \};
\end{align*}
are respectively the \emph{left annihilator} and \emph{annihilator} of $S$. In case that $S = A$, they are respectively called the \emph{socle} $\Soc(A)$ and the \emph{centre} $\zeta(A)$ of the left brace $A$ (see \cite{BonattoJedlicka23}, \cite{CatinoColazzoStefanelli19} or \cite{CedoSmoktunowiczVendramin19}, for example). It follows that a left brace $A$ is abelian if, and only if, $\zeta(A) = A$. Moreover, every subbrace of the centre of a left brace is an abelian ideal and every subbrace of the socle is abelian.

\begin{lema}[\cite{BallesterEstebanKurdachenkoPerezC-dedekind-leftbraces}]
\label{lema:left-right-ann_substruct}
Let $A$ be a left brace. For every ideal $I$ of $A$, $(\Ann_A^l(I),\cdot)$ and $(\Ann_A(I), \cdot)$ are normal subgroups of $(A,\cdot)$.
\end{lema}

The socle of a left brace $A$ gives rise to an ordinal \emph{socle series} of $A$, which is an ascending series of ideals
\begin{align*}
 0 = \Soc_0(A) \leq \Soc_1(A) \leq \ldots \leq \Soc_{\alpha}(A) \leq \Soc_{\alpha +1} \leq \ldots \Soc_{\gamma}(A),
\end{align*}
defined by the following rule: $\Soc_1(A) = \Soc(A)$ and, recursively,
\begin{align*}
 \Soc_{\alpha+1}(A)/\Soc_{\alpha}(A) &= \Soc(A/\Soc_\alpha(A)),\quad \text{for every ordinal $\alpha$,}\\
 \Soc_{\lambda}(A) & = \bigcup_{\mu < \lambda} \Soc_{\mu}(A),\quad  \text{for every limit ordinal $\lambda$.}
\end{align*}
The last term $\Soc_\infty(A) = \Soc_\gamma(A)$ of this series is called the \emph{hypersocle} of~$A$. It follows that $A$ is right nilpotent if, and only if, $\Soc_n(A) = A$ for some non-negative integer $n$ (see~\cite{CedoSmoktunowiczVendramin19}). In this case, $A$ is also said to have \emph{finite multipermutational level}. If $\Soc_\infty(A) = A$ then we say that $A$ is \emph{hypermultipermutational}.

\subsection*{Left braces and solutions of the YBE}

\noindent A solution of the YBE is a pair $(X,r)$, where $X$ is  a set and $r\colon X \times X \rightarrow X \times X$ is a map, with $r(x,y) = (\lambda_x(y), \rho_y(x))$ for all $x,y\in X$, such that
\begin{itemize}
\item $r_{12}r_{23}r_{12} = r_{23}r_{12}r_{23}$ (set-theoretical solution), where $r_{12}= (r\times \id_X)$ and $r_{23} = (\id_X \times r)$;
\item $r^2 = \id_{X\times X}$ (involutive);
\item $\lambda_x$ and $\rho_x$ are injective for every $x\in X$ (non-degenerate).
\end{itemize} 
A \emph{twist solution} $(X,r)$ is given by $r(x,y) = (y,x)$ for every $x\in X$, i.e. $\lambda_x = \rho_x = \id_X$ for every $x\in X$. 

The \emph{retraction solution} $\operatorname{Ret}(X,r) = (X/\!\sim,\overline{r})$ is defined by means of the equivalence relation $\sim$ in $X$: $x\sim y$ if $\lambda_x = \lambda_y$ and $\rho_x = \rho_y$. Then,
\[ \overline{r}\big([x], [y]\big) = \big([\lambda_x(y)], [\rho_y(x)]\big), \quad \text{for all $[x], [y] \in X/\! \sim$}.\]
The retraction series is defined recursively as
\begin{align*}
\operatorname{Ret}^0(X, r) & = (X,r)\\
\operatorname{Ret}^{\alpha+1}(X,r) & = \operatorname{Ret}\big(\operatorname{Ret}^\alpha(X,r)\big), \quad \text{for every ordinal $\alpha$,}\\
\operatorname{Ret}^{\lambda}(X,r) & = (X/\!\sim_{\lambda}, \overline{r}_{\lambda}),  \quad \text{for every limit ordinal $\lambda$,}
\end{align*}
where $\sim_{\lambda} := \bigcup_{\mu< \lambda} \sim_{\mu}$, and $\sim_\mu$ is the equivalence relation  in $\operatorname{Ret}^{\mu}(X,r)$ for every $\mu < \lambda$.

A solution $(X,r)$ is \emph{multipermutational of level} $m$, if $m$ is the smallest non-negative integer such that $\operatorname{Ret}^m(X,r)$ has cardinality $1$. For instance, twist solutions are multipermutational solutions of level~$1$. We say that $(X,r)$ is \emph{hypermultipermutational} if $\Ret^\infty(X,r)$ has cardinality $1$. In case that $X$ is a finite set, it is clear that a solution is hypermultipermutational if, and only if, it is multipermutational of finite level.

Following~\cite{Rump07}, every left brace $A$ provides a solution $(A, r_A)$, where 
\[ r_A(a,b) = \big(\lambda_a(b), \lambda^{-1}_{\lambda_a(b)}(a)\big) \quad \text{for every $a,b\in A$.}\]
We call $(A,r_A)$ the \emph{associated solution} with $A$. 

Conversely, if $(X,r)$ is a solution, with $r(x,y) = (\lambda_x(y),\rho_y(x))$ for every $x,y\in X$, the \emph{structure group} of the solution is defined as
\[ G(X,r):= \langle x\in X \mid xy = \lambda_x(y)\rho_y(x), \, \text{for every $x,y\in X$}\rangle\]
and it admits a left brace structure. We call $(G(X,r), + ,\cdot)$ the \emph{left brace structure} of $(X,r)$. In particular, if $(X,r)$ is a twist solution, then $G(X,r)$ is an abelian left brace isomorphic to $\mathbb{Z}^X$, the free abelian group on $X$.  Moreover, it holds that the following diagram commutes
\begin{equation}
\label{eq:diagram}
 \xymatrix{ X \times X \ar[r]^{r} \ar[d]_{\iota \times \iota} & X \times X \ar[d]^{\iota \times \iota} \\ G \times G \ar[r]_{r_G} & G \times G}
\end{equation}
where $G:= G(X,r)$, $(G, r_G)$ is the associated solution with $G$, and the canonical inclusion $\iota\colon X \rightarrow G(X,r)$ is an injective map (see~\cite{EtingofSchedlerSoloviev99}). Observe that the notation $\lambda$ for the first component of solutions does not suppose a notation clash with the $\lambda$-action as diagram~\eqref{eq:diagram} ensures that both coincide in~$X$.

This correspondence paves the way to characterise the multipermutational character of a solution by means of nilpotency of left braces. It follows that $(X,r)$ is (hyper)multipermutational if, and only if, $G(X,r)$ is (hyper)multipermutational (see~\cite{CedoJespersKubatVanAntwerpenVerwimp23} and~\cite{CedoSmoktunowiczVendramin19}).

\section{Main results}
\label{sec:main}

In this section we collect our main results about Dedekind left braces whose additive group is not periodic. The proof of each of them is presented in each corresponding subsection.

In~\cite{BallesterEstebanKurdachenkoPerezC-dedekind-leftbraces}, we define \emph{Dedekind left braces} as left braces such that every subbrace is an ideal. It is clear that the property of being Dedekind is inherited by subbraces and quotients. On the other hand, left braces whose additive group is isomorphic to the infinite cyclic group $\mathbb{Z}$ are a first natural example of left braces whose additive group is not periodic.

\begin{proposicio}
\label{prop:dedekindbraces-Z}
Abelian left braces isomorphic to $\mathbb{Z}$ are the unique Dedekind left braces whose additive group is isomorphic to an infinite cyclic group.
\end{proposicio}

\begin{proof}
It is clear that abelian left braces are Dedekind. On the other hand, it is proved in~\cite{Rump07-jpaa} that there are only two left braces, up to isomophism, whose additive group is isomorphic to $\mathbb{Z}$.

Let $A$ be left brace such that $(A, +) = \langle a \rangle_+ \cong \mathbb{Z}$ for some $a\in A$. If $A$ is not abelian, then $(A, \cdot)$ is defined as follows (see \cite{Rump07-jpaa}):
\[ (na) \cdot (ma) := \big((-1)^nm + n\big)a \quad \text{for every $n,m \in \mathbb{Z}$.}\] 
It follows that subbraces of $A$ are of the form $\langle ka \rangle_+$, with $k \in \mathbb{Z}$. Moreover, $\Soc(A) = \langle 2a \rangle_+$, as $(2a) \cdot (na) = 2a + na$ for every $n\in \mathbb{Z}$, and $(ma)\cdot (ma) = 0$ for every odd integer $m$. 

Let $m$ an odd integer and $k \in \mathbb{Z}$ such that $m = 2k+1$. Then, $\langle ma \rangle_+$ is not an ideal, since 
\begin{align*}
 (ma) \ast a & = \big((2k+1)a\big) \ast a = \big((2ka)\cdot a\big) \ast a = \\
& = (2ka) \ast (a \ast a) + a \ast a + (2ka) \ast a = a \ast a = a\cdot a - 2a = \\
& = -2a \notin \langle ma \rangle_+
\end{align*}
Hence, $A$ is not a Dedekind left brace.
\end{proof}

Proposition~\ref{prop:2-nil->centnil} states that $2$-right nil (or, equivalently, $2$-left nil) left braces are centrally nilpotent. In case of Dedekind left braces whose additive group is not periodic we can say more.

\begin{athm}
\label{ateo:2.11nou}
Let $A$ be a Dedekind left brace whose additive group is not periodic. If $a \ast a = 0$ for each element $a\in A$, then $A$ is abelian.
\end{athm}

\begin{corolari}
\label{cor:twistsolutions}
Let $(X,r)$ be a solution such that $G(X,r)$ is a Dedekind left brace. Assume that $r(x,x) = (x,x)$ for every $x\in X$. Then, $(X,r)$ is a twist solution and $G(X,r)$ is abelian isomorphic to $\mathbb{Z}^X$.
\end{corolari}

The next theorems deal with the hypermultipermutational case and the multipermutational case of level $2$, respectively. In the former case, bearing in mind that socles of left braces are canonical abelian subbraces, it turns out that the torsion-free character of socles is a key property to make Dedekind left braces abelian. The latter case relies on an exhaustive study of the structural relation between the first and  second non-zero terms of the socle series of Dedekind left braces.

\begin{athm}
\label{ateo:Ahypersoc_torsionfreee->Aabelian}
Let $A$ be a Dedekind hypermultipermutational left brace. If the additive group of the socle is torsion-free, then $A$ is abelian.
\end{athm}

\begin{corolari}
\label{cor:mult+dedekind->twist}
Let $(X,r)$ be a hypermultipermutational solution such that the left brace structure $G(X,r)$ is Dedekind whose additive group is not periodic. Then, $(X,r)$ is a twist solution. In particular, if $X$ is a finite set, multipermutational solutions $(X,r)$ with Dedekind left braces $G(X,r)$ are necessarily twist solutions.
\end{corolari}

\begin{athm}
\label{ateo:2.17nou}
Let $A$ be a Dedekind left brace such that $A = \Soc_2(A)$ and the additive group of $A$ is not periodic. Then $A$ is embedded into a direct sum $T \oplus D$ where
\begin{enumerate}
\item $T$ is a Dedekind left brace whose  additive group is periodic,
\item the additive group of $T\ast T$ is locally cyclic;
\item $\zeta(T)$ includes $T\ast T$;
\item $D$ is abelian.
\end{enumerate}

\end{athm}

\subsection{Proof of Theorem~\ref{ateo:2.11nou}}
\label{subsec:proof1}

\begin{proof}[Proof of Theorem~\ref{ateo:2.11nou}]
We work towards a proof of $A \ast A = 0$. Let $a , b \in A$ arbitrary elements of $A$.  Since both $a \ast a = 0 = b \ast b$, Lemma~\ref{lema:a·a=0ciclic} shows that $\langle a \rangle_+, \langle b \rangle_+$ are  ideals of $A$. Assume that $a$ has infinite additive order. If $b$ has finite additive order, then $\langle a \rangle_+ \cap \langle b \rangle_+ = \{0\}$. Thus, $a \ast b = b \ast a  = 0$. Suppose now that $b$ has infinite additive order and $\langle a \rangle_+ \cap \langle b \rangle_+ \neq \{0\}$. It follows that $na = kb$ for some integers $n,k$. Thus, applying Equation~\eqref{ast_distesq}, it holds
\begin{align*}
0 = n(a \ast a) = a \ast (na) = a \ast (na) = a \ast (kb) = k(a \ast b)
\end{align*}
But $a \ast b \in \langle a \rangle_+$. Therefore, $a \ast b = 0$. Similarly, we also get $b \ast a = 0$.

Assume now that $a$ has finite additive order. If $b$ has infinite additive order, by the previous paragraph $a \ast b = b \ast a = 0$. Suppose that $b$ has also finite additive order. Since $(A,+)$ is not periodic, there exists $c\in A$ such that $c$ has infinite additive order. Therefore, $b+c$ has infinite additive order, and then
\[ 0 = a \ast (b+c) = a \ast b + a \ast c = a \ast b + 0 = a \ast b\]
Similarly, we also get $b\ast a = 0$. 
\end{proof}

\begin{proof}[Proof of Corollary~\ref{cor:twistsolutions}]
According to diagram~\eqref{eq:diagram}, $r(x,x) = (x,x)$ for every $x\in X$ implies that $x \ast x = 0$ in the left brace structure associated $G(X,r)$. Now, we can argue as in the previous theorem so that $x\ast y = 0$, for every $x, y \in X$. Again, by diagram~\eqref{eq:diagram}, this implies that $r(x,y) = (y,x)$ for every $x,y\in X$. As a consequence, $G(X,r)\cong \mathbb{Z}^X$ and $G(X,r)$ is abelian.
\end{proof}

\subsection{Proof of Theorem~\ref{ateo:Ahypersoc_torsionfreee->Aabelian}}

\label{subsec:proof2}

We work towards a proof of Theorem~\ref{ateo:Ahypersoc_torsionfreee->Aabelian}. In the whole subsection $A$ denotes a left brace, $S$ denotes $\Soc(A)$ and $C$ denotes $\Soc_2(A)$. First, we need the following two key lemmas.

\begin{lema}
\label{lema:dastd-nozero+Dededkind-><d>=<d>+<dastd>}
Suppose that $A$ is a Dedekind left brace and let  $d \in C$ such that $d \ast d \neq 0$. Then, $\langle d \rangle = \langle d \rangle_+ + \langle d\ast d \rangle_+$.  
\end{lema}

\begin{proof}
Since $d\in C$, we have that $d\ast a \in S$ for every $a\in A$. In particular, $0 \neq u := d\ast d \in S$. By Lemma~\ref{lema:a·a=0ciclic}, it holds $U:= \langle u \rangle = \langle u \rangle_+$. Since $U$ is an ideal, using again Lemma~\ref{lema:a·a=0ciclic}, it follows $\langle d+U\rangle = \langle d+ U\rangle_+ = \langle d \rangle_+ + U$. Therefore, $\langle d \rangle  \leq \langle d \rangle_+ + U$. The other inclusion is trivial.
\end{proof}

\begin{lema}
\label{lema:CcapAnn_torsionfree+kxasta+torsionfree}
$C \cap \Ann_A(S)$ is a subbrace containing $S$. Moreover, it holds:
\begin{enumerate}
\item $N:= \Ann_C^l(S) = C \cap \Ann_A(S)$ is a normal subgroup of $(C,\cdot)$ such that either $N = C$ or $|C/N| = 2$ and $\lambda_a(u) = -u$, for every $a \in C \setminus S$ and every $u \in S$.
\item if $x \in C\cap \Ann_A(S)$, then  $x^k\ast a = k(x \ast a) = (kx) \ast a$ for every $a \in A$ and every non-negative integer $k\in \mathbb{Z}$;
\item if $(S,+)$ is torsion free, then $\big((C\cap \Ann_A(S))/S, +\big)$ is torsion-free.
\end{enumerate}
\end{lema}

\begin{proof}
By definition of $S = \Ann_A^l(A)$, it holds $\Ann_A^l(S) = \Ann_A(S)$ and $S \subseteq \Ann_A(S)$. By Lemma~\ref{lema:left-right-ann_substruct},  it follows $(\Ann_A(S), \cdot) \unlhd (A,\cdot)$, and therefore, $(C\cap \Ann_A(S), \cdot) \unlhd (A,\cdot)$ because $C$ is an ideal. Since $C/S$ is also abelian,  products and sums coincide in $C/S$, which yields that  $(C\cap \Ann_A(S))/S$ is a subbrace of $C/S$. Hence,  $C\cap \Ann_A(S)$ is a subbrace of $A$.

1. Clearly, $\Ann_C^l(S) = C \cap \Ann_A(S)$, since $S = \Ann_A^l(A)$. Thus, $N:= \Ann_C^l(S) = \Ker\bar{\lambda}$, where 
\[ \bar{\lambda}\colon a \in C \mapsto \left.\lambda_a\right|_S \in \Aut(S,+)\]
is a restriction of the $\lambda$-action to the ideal $S$. Therefore, $(C/N,\cdot)$ is isomorphic to a subgroup of $\Aut(S,+)$. Since $S$ is abelian, each subgroup of $S$ is a subbrace of $A$ and, therefore, an ideal. Hence, all subgroups of $(S,+)$ are $\bar{\lambda}$-invariant. Applying Theorem 1.5.7 of \cite{Schmidt94} we get that either $C= N$ or  $|C/N| = 2$ and $\lambda_a(u) = -u$, for every  $u\in S$ and every $a\in C\setminus S$. 

2. Let $x\in C\cap \Ann_A(S)$. Thus, $x \ast a \in S$ and $x \ast (x \ast a) = 0$. By induction, it follows that $x^k\ast a = k(x\ast a)$ for each non-negative integer $k$. Indeed, if $x^k \ast a = k(x \ast a)$ for some $k \geq 1$, then using equations~\eqref{ast_prod} and~\eqref{ast_distesq}, it holds
\begin{align*}
x^{k+1} \ast a & = (xx^k)\ast a = x\ast (x^k\ast a) + x^k \ast a + x \ast a = \\
& = x \ast \big(k(x \ast a)\big) + (k+1)(x \ast a) = k\big(x \ast (x \ast a)\big) + (k+1)(x\ast a) =\\
& = (k+1)(x\ast a).
\end{align*}
Moreover, there exists $s \in S$ such that $x^ks = kx$, as $x^kS = kxS$. Thus,
\[ (kx) \ast a = (x^ks) \ast a = x^k \ast (s \ast a) + s \ast a + x^k \ast a = x^k \ast a \]

3. Assume that there exists  $x\in C \cap \Ann_A^l(S) \setminus S$ and a positive integer $m$ such that $mx \in S$. Since $x\notin S$, there exists $a\in A$ such that $x \ast a \in S \setminus\{0\}$. By the previous paragraph, $m(x \ast a) = x^m \ast a = (mx) \ast a = 0$, which is a contradiction as $(S,+)$ is torsion-free.
\end{proof}

\begin{proof}[Proof of Theorem~\ref{ateo:Ahypersoc_torsionfreee->Aabelian}]
Assume that $A \neq S$. Since $A$ is hypermultipermutational, it follows that $C\neq S$.

Assume that $C \cap \Ann_A(S) \neq S$. Choose an element $d \in C \cap \Ann_A(S)$ such that $d \notin S$, and set $D = \langle d \rangle$. Since $d \in C$ and $D$ is an ideal by hypothesis, it holds $d\ast x \in D \cap S$ for every $x \in A$. If $D \cap S = 0$, then $d \in S$, a contradiction. Suppose that $D\cap S \neq 0$. If $d \ast d  = 0$, by Lemma~\ref{lema:a·a=0ciclic}, $D = \langle d \rangle_+$. Then, there exists a positive integer $k$ such that $kd \in S$. This contradicts the fact that $\big((C \cap \Ann_A(S))/S,+\big)$ is torsion free by Lemma~\ref{lema:CcapAnn_torsionfree+kxasta+torsionfree}. 

If $d \ast d \neq 0$, then by Lemma~\ref{lema:dastd-nozero+Dededkind-><d>=<d>+<dastd>}, $\langle d \rangle = \langle d \rangle_+ + \langle d \ast d \rangle_+$. Moreover, since $\langle d \ast d \rangle_+ \leq S$ and $\big((C\cap \Ann_A(S))/S,+\big)$ is torsion-free by Lemma~\ref{lema:CcapAnn_torsionfree+kxasta+torsionfree}, it follows that $\langle d \rangle_+ \cap \langle d \ast d \rangle_+ = 0$. Hence, $\langle d \rangle = \langle d \rangle_+ \oplus \langle d \ast d \rangle_+$. Take $c := 2d \in \langle d \rangle$. Then, $d \ast c = d \ast (2d) = 2(d \ast d)$. If $c \ast c = 0$, then $\langle c \rangle = \langle c\rangle_+ = \langle 2d\rangle_+$ by Lemma~\ref{lema:a·a=0ciclic}. Thus, $d \ast c\notin \langle c\rangle$, by the previous paragraph. If $c \ast c \neq 0$, then by Lemma~\ref{lema:CcapAnn_torsionfree+kxasta+torsionfree}, it holds
\[ c\ast c = (2d) \ast (2d) = 2((2d) \ast d) = 2((d^2) \ast d) = 4(d\ast d).\]
By the previous paragraph, $\langle c \rangle = \langle c \rangle_+\oplus \langle c\ast c \rangle_+ = \langle 2d\rangle_+ + \langle 4(d\ast d)\rangle_+$. Thus, $d\ast c \notin \langle c \rangle$. In both cases, it follows that $\langle c \rangle$ is not an ideal, which is a contradiction. Hence, $C\cap \Ann_A(S) = S$. 

Now, by Lemma~\ref{lema:CcapAnn_torsionfree+kxasta+torsionfree}, $|C/S| = 2$ and $\lambda_a(u) = -u$, for every  $u\in S$ and every $a\in C\setminus S$. Observe that the unique left brace up to isomorphism of order $2$ is abelian isomorphic to a cyclic group of order~$2$. In particular, if $a \in C \setminus S$, then $a^2 \in S$ and  $-a^2 = \lambda_a(a^2) = aa^2 - a = a^3 - a$.  Since $a^2 \in S$, it follows that $a = a^3 + a^2 = a^2 + a^3 = a^2a^3 = a^5$. Thus,  $a^4 = 0$. Again, using $a^2 \in S$, it holds $0 = a^4 = a^2a^2 = a^2 + a^2 = 2a^2$, which is a contradiction as $(S,+)$ is torsion free. 

This contradiction yields $A = S$, and therefore, $A$ is abelian.
\end{proof}

\begin{proof}[Proof of Corollary~\ref{cor:mult+dedekind->twist}]
We can apply Theorem~\ref{ateo:Ahypersoc_torsionfreee->Aabelian} to $G(X,r)$ as it is torsion-free. Therefore, $G(X,r)$ is abelian, which means that $\lambda_x = \rho_x = \id_X$ for every $x\in X$.
\end{proof}

\subsection{Proof of Theorem~\ref{ateo:2.17nou}}

\label{subsec:proof3}

In this subsection Equations~\eqref{ast_distesq} and~\eqref{ast_prod} are used all over the proofs. Lemma~\ref{lema:dastd-nozero+Dededkind-><d>=<d>+<dastd>} of previous subsection play also a key role.

We split the proof of Theorem~\ref{ateo:2.17nou} in different steps. Our first goal is to show that for Dedekind left braces the additive group of the quotient 
between the second and the first non-zero term of the socle series is periodic (see Proposition~\ref{prop:2.4nou}). To this aim, the following lemmas deal with additive orders of elements in a left brace~$A$.

\begin{lema}
\label{lema:infiniteorder->subbridanoideal}
Let $a\in A$ with infinite additive order. Suppose that $0 \neq c:= a \ast a$ satisfies $c \ast c = a \ast c = c \ast a = 0$. If $c$ has infinite additive order, then $\langle a \rangle = \langle a \rangle_+ \oplus \langle c \rangle_+$ and it includes a subbrace which is not an ideal.
\end{lema}

\begin{proof} 
From the hypothesis, it is not hard to show that $\langle a \rangle = \langle a \rangle_+ + \langle c \rangle_+$. 

Suppose that  $\{0\} \neq \langle a \rangle_+ \cap \langle c \rangle_+$. If $\langle c \rangle_+ \leq \langle a \rangle_+$, then $c = ta$ for some positive integer $t$. Thus, $tc = t(a \ast a) = a \ast (ta) = 0$, which is not possible. Similarly, if $\langle a \rangle_+ \leq \langle c \rangle_+$, we arrive to $a \ast a = 0$, a contradiction. Therefore, $\langle a \rangle_+ \neq \langle a \rangle_+ + \langle c \rangle_+ \neq \langle c \rangle_+$. Since $\{0\} \neq \langle a \rangle_+ \cap \langle c \rangle_+$, the $0$-rank of $\langle a \rangle_+ + \langle c \rangle_+$ is $1$, that is $ \langle a \rangle = \langle a \rangle_+ + \langle c \rangle_+ = \langle d \rangle_+ \oplus \langle u \rangle_+$ where $d$ has infinite additive order and $u$ has finite additive order $n$. Let $k, j \in \mathbb{Z}$ such that  $u = ka + jc$. It holds
\[ a \ast u = a \ast (ka+jc) = a \ast (ka) + a \ast (jc) = a \ast (ka) = k(a\ast a) = kc.\]
It follows $0 = a \ast 0 = a \ast(nu) = n(a\ast u) = nkc$, a contradiction. 
which contradicts the additive order of $c$. Hence, $\{0\} = \langle a \rangle_+ \cap \langle c \rangle_+$, and then, $\langle a \rangle = \langle a \rangle_+ \oplus \langle c \rangle_+$. 

Take $w := a^2 = 2a + a \ast a = 2a + c \in \langle a \rangle$, which has infinite additive order. It follows
\begin{align*}
c_1:= w\ast w & = a^2 \ast (2a + c) = 2(a^2 \ast a) + a^2 \ast c =\\
& = 2\big(a \ast (a \ast a) + 2(a \ast a)\big) + 0  =  4(a \ast a) = 4c.
\end{align*}
Moreover, since $c\ast c = 0$, it holds $4c = c^4$. Thus,  $c_1 \ast c_1 = c^4 \ast c^4 = 0$, $w\ast c_1 = a^2 \ast (4c) = 0$, and $c_1 \ast w = c^4\ast a^2 = 0$ hold. Therefore, by the previous paragraph, we can also conclude that $\langle w \rangle = \langle w \rangle_+ \oplus \langle c_1 \rangle_+$, so that every element of $\langle w\rangle$ can be written uniquely as 
\[ t_1w + t_2c_1 = t_1a^2 + t_24c = t_1(2a + c) + t_24c = 2t_1a + (4t_2 + t_1)c,\]
for certain integers $t_1, t_2$. 

Now, observe that $a\ast w = a \ast (a^2) = a \ast (2a + c) = 2(a \ast a) = 2c$, i.e.  $a\ast w \notin\langle w \rangle$. Hence, it follows that $\langle w \rangle$ is not an ideal. 
\end{proof}

\begin{lema}
\label{lema:dinSoc2+order-inf+dastd=0->dinS}
Let $A$ be a Dedekind left brace and let $d\in \Soc_2(A)$ with infinite additive order. If $d \ast d = 0$, then $d \in \Soc(A)$.
\end{lema}

\begin{proof}
By Lemma~\ref{lema:a·a=0ciclic}, it holds $\langle d \rangle = \langle d \rangle_+$. Since it is an ideal, we have that $(\langle d \rangle_+, \cdot)$ is a normal subgroup of $(A,\cdot)$ and $\langle d \rangle_+$ is $\lambda$-invariant. Therefore, for every $x\in A$, either $x^{-1}dx = d$ or $x^{-1}dx = -d$, and either $\lambda_x(d) = d$ or $\lambda_x(d) = -d$.

Let $x \in A$ and assume $xd = dx$. It follows $d \ast x = dx -d - x = xd -x  -d = x \ast d$. If $\lambda_x(d) = d$, then $d \ast x = x \ast d = \lambda_x(d) - d = 0$. If $\lambda_x(d) = - d$, then $d \ast x = x \ast d = -2d \in \Soc(A)$. Thus, $2d = d^2 \in \Soc(A)$. Therefore, it follows that $d^2x = xd^2$ and $x\ast d^2 = d^2\ast x = 0$, i.e. $\lambda_x(d^2) = d^2$. But then, we see that
\begin{align*}
2d = d^2 = \lambda_x(d^2) = \lambda_x(2d) = 2\lambda_x(d) = -2d,
\end{align*}
which is a contradiction.

Assume now $dx = x(-d)$. Then, $dx = x(0-d) = 2x - xd$. If $\lambda_x(d)  = d$, then $dx = 2x - (x + \lambda_x(d)) = x - d$. Thus, $d \ast x = dx -d -x = -2d \in \Soc(A)$. Then, $2d = d^2 \in \Soc(A)$. It follows that
\[ 0 = d^2 \ast x =  d \ast (d \ast x) + 2(d\ast x) \]
Since $d \ast x \in \langle d \rangle = \langle d \rangle_+$ and $\langle d \rangle_+$ is an abelian subbrace, it follows that $d \ast (d \ast x) = 0$. Hence, $2(d\ast x) = 0$, which is a contradiction as $\langle d \rangle_+$  is isomorphic to an infinite cyclic group.

It remains the case, $dx = x(-d)$ and $\lambda_x(d) = -d$. Then, $dx = 2x - (x-d) = x+d$, and therefore, $d\ast x = dx -d -x = 0$. Hence, we can conclude $d\ast x = 0$ for every $x\in A$, i.e. $d \in \Soc(A)$.
\end{proof}

\begin{lema}
\label{lema:2.3nou}
Let $A$ be a Dedekind left brace and let $d\in \Soc_2(A)$ with infinite additive order. Then, $d \ast d$ has finite additive order.
\end{lema}

\begin{proof}
Since $d \in \Soc_2(A)$, it follows that $c:= d \ast d \in \Soc(A)$. Suppose that $c$ has infinite additive order. 
Since $\Soc(A)$ is abelian, it follows that $\langle c \rangle = \langle c \rangle_+$ is a subbrace of $A$, and therefore, an ideal. Thus, $(\langle c \rangle_+,\cdot)$ is a normal subgroup of $(A, \cdot)$, so that either $dcd^{-1} = c$ or $dcd^{-1} = -c = c^{-1}$.

Assume $dcd^{-1} = c$. Since $c\in \Soc(A)$, $c \ast c = c \ast d = 0$. Moreover, $d \ast c = dc -d -c = cd -c - d = c\ast d = 0$. Applying, Lemma~\ref{lema:infiniteorder->subbridanoideal}, there exists a subbrace which is not an ideal; a contradiction.

Assume $dcd^{-1} = -c$. Then, it holds that 
\begin{align*}
d \ast c & = dc -d - c = -c + (-c)d -d  = -2c + (-c)d + c - d =\\
& = -2c + (-c)\ast d = -2c.
\end{align*}
On the other hand, since $\Soc_2(A)/\Soc(A)$ is abelian, $2d = d^2s$ for some $s \in \Soc(A)$. Thus, it also holds
\begin{align*}
(2d) \ast (2d) & = (d^2s) \ast (2d) = d^2 \ast \big(s \ast (2d)\big) + s \ast (2d) + d^2 \ast (2d) = \\
& =  d^2 \ast (2d) = 2(d^2 \ast d) = 2\big(d \ast (d \ast d) + 2d \ast d\big) = \\
& = 2(d \ast c + 2c) = 2(-2c + 2c) = 0
\end{align*}
Therefore, by Lemma~\ref{lema:dinSoc2+order-inf+dastd=0->dinS}, $2d\in \Soc(A)$. Hence, $2d + c \in \Soc(A)$ and 
\[ d \ast (2d + c) = 2d \ast d + d \ast c = 0.\]

By Lemma~\ref{lema:dastd-nozero+Dededkind-><d>=<d>+<dastd>}, we can write $\langle d \rangle = \langle d \rangle_+ + \langle c \rangle_+ = \langle d \rangle_+ + \langle 2d + c \rangle_+$. Assume that $\langle d \rangle = \langle d \rangle_+ \oplus \langle 2d + c\rangle_+$. Since $4d+c\in \Soc(A)$, then $\langle 4d +c \rangle = \langle 4d+c \rangle_+$ is an ideal. However, we see that
\[ d \ast (4d + c) = 4(d \ast d) + d \ast c = 4c - 2c = 2c \notin \langle 4d + c \rangle_+ = \langle 2d + 2d+c\rangle_+\]
Hence, we arrive to a contradiction.

Assume that $\langle d \rangle_+ \cap \langle 2d+c\rangle_+ \neq \{0\}$. If $\langle 2d + c \rangle_+ \leq \langle d \rangle_+$, then $\langle d \rangle = \langle d \rangle_+$ is a Dedekind left brace whose additive group is isomorphic to $\mathbb{Z}$. Thus, Proposition~\ref{prop:dedekindbraces-Z} yields that $\langle d \rangle$ is abelian, i.e. $d \ast d = 0$, which is a contradiction. If $\langle d \rangle_+ \leq \langle 2d +c \rangle_+$, then $d \in \Soc(A)$, a contradiction. Therefore, 
\[ \langle d \rangle_+ \neq \langle d \rangle_+ + \langle 2d+c \rangle_+ \neq \langle 2d+c \rangle_+.\] Since $\{0\} \neq \langle d \rangle_+ \cap \langle 2d+c \rangle_+$, the $0$-rank of $\langle d \rangle$ is $1$. Thus, $ \langle d \rangle = \langle v \rangle_+ \oplus \langle u \rangle_+$ where $v$ has infinite additive order and $u$ has finite additive order $n$. Let $l,m \in \mathbb{Z}$ such that  $u = ld + m(2d+c)$. It follows that
\[ d \ast u = d \ast \big(ld+ m(2d+c)\big) = ld \ast d = lc.\]
On the other hand,  $0 = d \ast 0 = d \ast(nu) = n(d\ast u) = nlc$, i.e. $c$ has finite additive order. We arrive to a final contradiction. 
\end{proof}

\begin{proposicio}
\label{prop:2.4nou}
Let $A$ be a Dedekind left brace. The additive group of the quotient $\Soc_2(A)/\Soc(A)$ is periodic.
\end{proposicio}

\begin{proof}
Set $S:= \Soc(A)$ and $C:= \Soc_2(A)$. Assume that the additive group of $C$ contains an element $d$ such that $zd \notin S$ for all integers $z$, i.e. $d$ has infinite additive order. In particular, Lemma~\ref{lema:dinSoc2+order-inf+dastd=0->dinS} implies that $c = d\ast d \neq 0$. Moreover, by Lemma~\ref{lema:2.3nou}, $c$ has a finite additive order in $A$. Thus, according to Lemma~\ref{lema:dastd-nozero+Dededkind-><d>=<d>+<dastd>}, we can write $\langle d \rangle = \langle d \rangle_+ \oplus \langle c \rangle_+$, where $\langle c \rangle_+$ is a finite subbrace, and therefore, a finite ideal. Hence, $\langle c \rangle_+$ is a finite normal subgroup of $(A, \cdot)$. Thus, there is a positive integer $n$ such that $d^{-n}cd^n = c$. 

Now, it follows that $nd = d^nx$ for some  $x\in S$ as $C/S$ is abelian. Then, we get that
\[ (nd) \ast c = (d^nx)\ast c = d^n \ast (x \ast c) + x \ast c + d^n \ast c = d^n \ast c.\]
Moreover, since $cd^n = d^nc$, we have that $(nd) \ast c = d^n \ast c = c\ast d^n = 0$. 

Set $v := nd$ with infinite additive order. Then, Lemma~\ref{lema:dinSoc2+order-inf+dastd=0->dinS} implies that $u:= v \ast v \neq 0$ and so, $\langle v \rangle = \langle v \rangle_+ + \langle u \rangle_+$ by Lemma~\ref{lema:dastd-nozero+Dededkind-><d>=<d>+<dastd>}. It follows that
\begin{align*}
u & = v \ast v  = (nd) \ast (nd) = n\big((nd) \ast d\big) = n\big((d^nx)\ast d\big) = \\
& = n\big(d^n \ast (x \ast d) +  x \ast d + d^n \ast d\big)  = n(d^n \ast d).
\end{align*}
We claim that $d^m \ast d \in \langle c \rangle_+$ for every positive integer $m$. Indeed, if $d^m \ast d \in \langle c \rangle_+$ for a certain $m \geq 1$, it holds
\begin{align*}
d^{m+1}\ast d & = (dd^m)\ast d = d \ast (d^m\ast d) + d^m\ast d + d \ast d = \\
& = d \ast (d^m\ast d) + d^m\ast d + c \in \langle c \rangle_+,
\end{align*}
as $\langle c \rangle_+$ is an ideal. In particular, it follows that $u\in \langle c \rangle_+$, so that $u = jc$ for some integer $j$. Thus, $u$ has finite  additive order $k$.

Now, $v \ast u = j(v\ast c) = 0 = u \ast v$, as $u \in S$. We claim $(tv) \ast v = t(v \ast v) = tu$, for every positive integer $t$. Indeed, suppose that $(tv) \ast v = tu$. Since $(t+1)v + S = v(tv)+ S = v(tv)S$, then $(t+1)v = v(tv)y$, for some $y\in S$. Therefore, it holds
\begin{align*}
 \big((t+1)v\big) \ast v & = \big(v(tv)y\big) \ast v = \big(v(tv)\big) \ast (y \ast v) + y\ast v + \big(v(tv)\big) \ast v  = \\
& = \big(v(tv)\big) \ast v = v \ast \big((tv)\ast v\big) + (tv) \ast v + v \ast v = \\
& =  v \ast (tu) + tu + u = t(v \ast u) + (t+1)u = (t+1)u.
\end{align*}

In particular, it follows that $(kv) \ast (kv) = k\big((kv) \ast v\big) = k^2u = 0$.  Moreover, $(kv) = knd \in \langle d \rangle_+$ has infinite additive order. Hence,  Lemma~\ref{lema:dinSoc2+order-inf+dastd=0->dinS} yields $knd \in S$, which is a contradiction. 
\end{proof}

The second key step is to show that for a Dedekind left brace $A$, if $A = \Soc_2(A)$ has non-periodic additive group then the socle coincides with the centre of $A$. To this aim, we prove that if $\Soc_2(A)$ has non-periodic additive group then $\Soc_2(A)$ is contained in $\Ann_A(\Soc(A))$ (see Proposition~\ref{prop:2.14nou}).

\emph{From now on, $A$ denotes a Dedekind left brace.}

\begin{lema}
\label{lema:2.5nou}
Let $d \in \Soc_2(A)$ with infinite additive order. Then $d\ast x, x \ast d \in \langle d \ast d \rangle_+$, for every $x \in \Ann_A(\Soc(A))$.
\end{lema}

\begin{proof} 
Let $x \in \Ann_A(\Soc(A))$. Proposition~\ref{prop:2.4nou} implies that there is a positive integer $n$ such that $nd \in \Soc(A)$. Then, $(nd)\ast x = x \ast (nd) = 0$. In particular, it follows that $\lambda_x(nd) = nd$. By Lemma~\ref{lema:dastd-nozero+Dededkind-><d>=<d>+<dastd>}, $\langle d \rangle = \langle d \rangle_+ + \langle d \ast d \rangle_+$.  Moreover $U:= \langle d \ast d \rangle_+$ is a subbrace as $d\ast d \in \Soc(A)$. Thus, $U$ is an ideal of $A$ and Lemma~\ref{lema:2.3nou} implies that $U$ is finite. Therefore, $\langle d \rangle = \langle d\rangle_+ \oplus U$.

The quotient $A/U$ is also a Dedekind left brace. Then, $(d+U) \ast (d+U) = (d\ast d) + U = U$. Thus, $ \langle d+ U \rangle = \langle d + U \rangle_+ = \langle d\rangle_+ + U$ is a subbrace which is an ideal in $A/U$.  Moreover, $d+U$ has additive infinite order, and therefore,  Lemma~\ref{lema:dinSoc2+order-inf+dastd=0->dinS} implies $d+U \in \Soc(A/U)$.  In particular, it follows that 
\[ U = (d+U)\ast (x+U) = (d\ast x) + U,\]
so that $(d\ast x ) \in U$. 

On the other hand, since $\langle d +U\rangle_+$ is an ideal of $A/U$, we have that  $\lambda_{x+U}(\langle d + U \rangle_+) = \langle d + U \rangle_+$. Thus, either $\lambda_{x + U}(d+U) = d+U$ or $\lambda_{x+U}(d+U) =  -d+U$. In the former case, we obtain
\[ d+U = (x+U)(d+U) - (x+U) = xd - x + U.\]
Therefore, $x\ast d = xd - x -d \in  U$.

Suppose now that $\lambda_{x+U}(d+U) = -d + U$. Then, $\lambda_{x+U}(nd+U) = -nd+U$. On the other hand, since $\lambda_x(nd) = nd$, we obtain that $\lambda_{x+U}(nd+U) = nd+U$ which yields that $2nd + U = U$. Since $U$ is finite, this implies that $d$ has finite additive order, which is a contradiction.
\end{proof}

\begin{lema}
\label{lema:2.6nou}
Let $d \in \Soc_2(A)$ with finite additive order. If the additive group of $\Soc_2(A)$ is not periodic, then $d \ast x, x \ast d \in \langle d \ast d \rangle_+$ for every  $x\in \Ann_A(\Soc(A))$.
\end{lema}

\begin{proof}
Let $k$ be the additive order of $d$. Then, $0 = d \ast (kd) = k(d\ast d)$. Thus, by Lemma~\ref{lema:dastd-nozero+Dededkind-><d>=<d>+<dastd>}, $\langle d \rangle = \langle d \rangle_+ + \langle d\ast d \rangle_+$ is finite. On the other hand, Proposition~\ref{prop:2.4nou} implies that the additive group of $\Soc(A)$ is not periodic. Thus, there exists $b \in \Soc(A)$ with infinite additive order. Therefore, $\langle b \rangle_+ \cap \langle d \rangle = \{0\}$. Since both $\langle b\rangle_+$ and $\langle d \rangle$ are ideals, we have that $d \ast b = 0 = b \ast d$. It follows that
\begin{align*}
(b+d) \ast (b+d) & = (bd) \ast (b+d) = (bd) \ast b + (bd) \ast d = \\
&= b \ast (d \ast b) + b\ast b + d \ast b + b \ast (d\ast d) + b \ast d + d\ast d = \\
& = d\ast d,
\end{align*}

Let $x\in \Ann_A(\Soc(A))$. Since $b+d$ has infinite additive order, Lemma~\ref{lema:2.5nou} implies that $(b+d)\ast x$, $x \ast (b+d) \in \langle d \ast d\rangle_+$. But we see that
\begin{align*}
& (b+d) \ast x = (bd) \ast x = b \ast (d \ast x) + b\ast x + d \ast x = d \ast x, \\
& x \ast (b+d) = x \ast b + x \ast d = x \ast d,
\end{align*}
as $b \in\Soc(A)$ and $x\in \Ann_A(\Soc(A))$. Hence, the lemma holds.
\end{proof}

\begin{proposicio}
\label{prop:2.14nou}
Suppose that the additive group of $\Soc_2(A)$ is not periodic. Then $\Soc_2(A) \leq \Ann_A(\Soc(A))$. In particular, if $A = \Soc_2(A)$ then $\Soc(A) = \zeta(A)$.
\end{proposicio}

\begin{proof}
Set $N := \Soc_2(A) \cap \Ann_A(\Soc(A))$ which is a subbrace by Lemma~\ref{lema:CcapAnn_torsionfree+kxasta+torsionfree}, and therefore an ideal of $A$. Suppose that $B\neq \Soc_2(A)$. Then, Lemma~\ref{lema:CcapAnn_torsionfree+kxasta+torsionfree} implies that $|\Soc_2(A)/N| = 2$ and $\lambda_u(a)= -a$ for every $u \in \Soc(A)$ and every $d\in \Soc_2(A)\setminus N$. Fix $d\in \Soc_2(A)\setminus N$. Since every left brace of order $2$ is abelian, it follows that $d^2, 2d\in N$. 

Assume that $d$ has infinite additive order. Then, Proposition~\ref{prop:2.4nou} implies that there is a positive integer $t$ such that $t(2d) \in \Soc(A)$. It follows that $\lambda_d(2td) = -2td$. On the other hand,
\[ \lambda_a(2td) = -d + d(2td) = -d + 2td^2 - (2t -1)d  = 2td^2 - 2td.\]
Thus, $0 = 2td^2 = 2t(2d + d\ast d) = 4td + 2(d\ast d)$. But Lemma~\ref{lema:2.3nou} implies that $d \ast d$ has finite additive order. Therefore, $4td$ has finite additive order, a contradiction. Hence, $d$ has finite additive order $m$, and then, 
\[ 0 = d \ast 0 = d \ast (md) = m(d \ast d).\]

Now, since the additive group of $\Soc_2(A)$ is not periodic and $\Soc_2(A) = \langle d \rangle_+ + N$, it follows that the additive group of $B$ is not periodic. Let $u\in N$ having infinite additive order. Then, by Lemma~\ref{lema:2.6nou}, $d \ast u \in \langle d \ast d\rangle_+$, and therefore, $d \ast u$ has finite additive order. On the other hand, by Proposition~\ref{prop:2.4nou}, there is a positive integer $n$ such that $nu \in \Soc(A)$. It follows that $\lambda_d(nu) = -nu$. But then, it holds that
\[ n(d\ast u) = d \ast (nu) = \lambda_d(nu) - nu = -2nu,\]
which implies that $u$ has finite additive order. This is a final contradiction.
\end{proof}

The last key step is to prove that if $\Soc_2(A)$ has non-periodic additive group then $\Soc(A)$ includes an abelian ideal isomorphic to a periodic and locally cyclic group, whose quotient is an abelian brace (see Corollaries~\ref{cor:2.13nou} and~\ref{cor:2.15nou}).

\emph{From now on, we assume that $\Soc_2(A)$ has non-periodic additive group.}

\begin{lema}
\label{lema:2.7nou}
Let $T$ be the periodic part of the additive group of $\Soc(A)$. Then, the quotient $\Soc_2(A)/T$ is abelian.
\end{lema}

\begin{proof}
If the additive group of $\Soc(A)$ is periodic, then $T = \Soc(A)$,  and therefore, $\Soc_2(A)/T = \Soc_2(A)/\Soc(A)$, which is abelian.

Suppose that the additive group of $\Soc(A)$ is not periodic. Since $\Soc(A)$ is abelian, we have that $T$ is  a subbrace of $A$, and therefore, an ideal of $A$. On the other hand, $\Soc_2(A) \leq \Ann_A(\Soc(A))$ by Proposition~\ref{prop:2.14nou}. 

Let $d\in \Soc_2(A)$ with finite additive order. Then, $d \ast d \in T \leq \Soc(A)$ and by Lemma~\ref{lema:2.6nou}, we have that $d\ast x, x \ast d \in \langle d \ast d \rangle_+ \leq T$ for every $x \in \Soc_2(A)$. 

Suppose now that $d\in \Soc_2(A)$ has infinite additive order. By Lemmas~\ref{lema:dastd-nozero+Dededkind-><d>=<d>+<dastd>} and~\ref{lema:2.3nou}, we have that $\langle d \rangle = \langle d \rangle_+ \cap \langle d \ast d \rangle_+$, where $d \ast d \in \Soc(A)$ has finite additive order. Thus, $\langle d \ast d \rangle_+\leq T$. Now, Lemma~\ref{lema:2.5nou} implies that $d\ast x, x\ast d \in \langle d \ast d \rangle_+ \leq T$ for every $x\in \Soc_2(A)$.
\end{proof}

\begin{lema}
\label{lema:2.9nou}
Assume that the periodic part of the additive group of $\Soc(A)$ is a $p$-subgroup for some prime $p$. If $a, b \in \Soc_2(A)$ satisfy $a \ast a \neq 0 \neq b \ast b$, then either $\langle a \ast a \rangle_+ \leq \langle b \ast b \rangle_+$ or $\langle b \ast b \rangle_+ \leq \langle a \ast a \rangle_+$.
\end{lema}

\begin{proof}
Set $u := a \ast a$, $v:= b \ast b \in \Soc(A)$. Lemma~\ref{lema:2.7nou} shows that $u$ and $v$ have finite additive orders.

Firstly, we claim that $\langle u \rangle_+ \cap \langle v \rangle_+ \neq \{0\}$. Indeed, suppose that $\{0\} = \langle u \rangle_+ \cap \langle v \rangle_+$. By Lemmas~\ref{lema:2.5nou} and~\ref{lema:2.6nou}, $a\ast b, b\ast a \in \langle u \rangle_+ \cap \langle v \rangle_+ = \{0\}$. In particular, $a+b = ab$. It follows that
\begin{align*}
(a+b) \ast (a+b) & = (ab) \ast (a+b) = (ab)\ast a + (ab) \ast b = \\
& = a \ast (b\ast a) + b\ast a + a \ast a + a \ast (b\ast b) + a \ast b + b \ast b = \\
& = a \ast a + b \ast b = u + v.
\end{align*}
Again, by Lemmas~\ref{lema:2.5nou} and~\ref{lema:2.6nou}, $a \ast (a+b) \in \langle (a+b) \ast(a+b)\rangle_+ = \langle u + v \rangle_+$. But, we see that $a\ast (a+b) = a \ast a + a \ast b = a \ast a = u$. Hence, we arrive to a contradiction, as  $\langle u \rangle_+ \cap \langle v \rangle_+ = \{0\}$ and $\langle u \rangle_+, \langle v \rangle_+$ are $p$-subgroups.

Set $U := \langle u \rangle_+ \cap \langle v \rangle_+$ which is a subbrace of $\Soc(A)$, and therefore an ideal of $A$. Suppose that neither $U \neq \langle u \rangle_+$ nor $U \neq \langle v \rangle_+$. It follows that both quotients $ \langle u \rangle_+/U$ and $\langle v \rangle_+/U$ are non-trivial. Then, we see that
\begin{align*}
& (a+U)\ast (a+U) = a \ast a + U = u + U \neq U, \\
& (b+U) \ast (b+U) = b\ast b + U = v + U \neq U.
\end{align*}
The hypothesis hold also for the quotient $A/U$, and therefore, by the first claim $\langle u +U \rangle_+ \cap \langle v + U \rangle_+$ is non-zero. This is a contradiction with the definition of~$U$.
\end{proof}

\begin{proposicio}
\label{prop:2.12nou}
Suppose that the periodic part of the additive group of $\Soc(A)$ is a $p$-subgroup for some prime $p$. Then, the following hold
\begin{enumerate}
\item $\Soc(A)$ contains an abelian ideal $P$ isomorphic to either a cyclic $p$-group or a Prüfer $p$-group;
\item the quotient $\Soc_2(A)/P$ is an abelian brace.
\end{enumerate}
\end{proposicio}

\begin{proof}
If $a \ast a = 0$ for every element $a \in \Soc_2(A)$, then Theorem~\ref{ateo:2.11nou} shows that $\Soc_2(A)$ is an abelian subbrace and the result holds. Therefore, suppose that $\Soc_2(A)$ has an element $a_1$ such that $0 \neq c_1:= a_1 \ast a_1 \in \Soc(A)$. By Lemma~\ref{lema:2.3nou}, there exists some positive integer $k_1$ such that $p^{k_1}c_1 = 0$. Moreover, $\langle c_1\rangle_+$ is an ideal of $A$, since $\langle c_1\rangle_+$ is a subbrace of $\Soc(A)$.

If $a \ast a \in \langle c_1\rangle_+$ for every element $a \in \Soc_2(A)$, then Theorem~\ref{ateo:2.11nou} shows that $\Soc_2(A)/\langle c_1\rangle_+$ is an abelian brace and we are done. Therefore, suppose that $\Soc_2(A)$ has an element $a_2$ such that $c_2:= a_2 \ast a_2  \notin \langle c_1 \rangle_+$. Again,  Lemma~\ref{lema:2.3nou} implies that $c_2$ has finite additive order, so that $p^{k_2}c_2 = 0$, for some positive integer $k_2$. Moreover, $\langle c_2\rangle_+$ is also an ideal of $A$, and  Lemma~\ref{lema:2.9nou} implies that $\langle c_1\rangle_+ \leq \langle c_2\rangle_+$, so that $k_1 < k_2$. 

Repeating recursively the above arguments we obtain that either $\Soc_2(A)$ has an element $a_t$ such that $\langle a_t \ast a_t\rangle_+ \leq \Soc(A)$ is isomorphic to a cyclic $p$-group and the quotient $\Soc_2(A)/ \langle a_t \ast a_t\rangle_+$ is abelian, or $\Soc_2(A)$ contains an infinite subset $\{a_j\mid j\in \mathbb{N}\}$ of elements such that, for every $j\in \mathbb{N}$, $c_j:= a_j\ast a_j \in \Soc(A)$ has finite additive order $p^{k_j}$, for some positive integer $k_j$. In the latter case, for every $j \in \mathbb{N}$, it also holds that $\langle c_j \rangle_+$ is an ideal such that $c_{j+1}\notin \langle c_j\rangle_+$, $\langle c_j \rangle_+ \leq \langle c_{j+1}\rangle_+$, and $p^{k_j} < p^{k_{j+1}}$. Set $P := \bigcup_{j\in \mathbb{N}} \langle c_j\rangle_+ \leq \Soc(A)$. Clearly, $P$ is an abelian ideal isomorphic to a Prüfer $p$-group.

Assume that there exists $d \in \Soc_2(A)$ such that $d \ast d \notin P$. By Lemma~\ref{lema:2.3nou}, $d\ast d$ has finite additive order $p^m$, for some positive integer~$m$. Choose $r \in \mathbb{N}$ such that $m < r$. By Lemma~\ref{lema:2.9nou}, we have that either $\langle d\ast d \rangle_+ \leq \langle c_r \rangle_+$ or $\langle c_r \rangle_+ \leq \langle d \ast d \rangle_+$. Since $d \ast d \notin P$, it follows that $\langle c_r \rangle_+ \leq \langle d \ast d \rangle_+$, which is a contradiction with the choice of $r$.  

Hence, we conclude that $a \ast a \in P$ for every $a \in \Soc_2(A)$. Therefore, Theorem~\ref{ateo:2.11nou} yields that $\Soc_2(A)/P$ is an abelian brace.
\end{proof}

\begin{corolari}
\label{cor:2.13nou}
It follows that $\Soc(A)$ includes an abelian ideal $P$ which is isomorphic to a periodic and locally cyclic group. Furthermore, the quotient $\Soc_2(A)/P$ is an abelian brace.
\end{corolari}

\begin{proof}
Let $T$ be the periodic part of the additive group of $\Soc(A)$. We have that $T = \bigoplus_{p \in \pi(T)} T_p$ where $T_p$ is the $p$-component of $T$, $p \in \pi(T)$. Fix $p\in \pi(T)$ and set $R_p := \bigoplus_{q\neq p} T_q$ so that $T = T_p \oplus R_p$. Since $\Soc(A)$ is abelian, $R_p$ is a subbrace, and therefore, an ideal of $A$. By Proposition~\ref{prop:2.12nou}, $T/R_p$ includes an abelian ideal $B_p/R_p$ which is isomorphic to either a cyclic $p$-group or a Prüfer $p$-group. Moreover, it also holds $\Soc_2(A)/R_p \ast \Soc_2(A)/R_p \leq B_p/R_p$. Therefore, we have that $B_p = C_p \oplus R_p$ where $C_p \leq \Soc(A)$ is an abelian ideal isomorphic to a cyclic $p$-subgroup or a Prüfer $p$-group. 

Now, set $P:= \bigoplus_{p \in \pi(T)} C_p \leq \Soc(A)$. Then, $P$ is an abelian ideal isomorphic to a periodic locally cyclic subgroup and $C_p \oplus R_p = P + R_p$ for every $p \in \pi(T)$. Thus, it also holds that $\Soc_2(A) \ast \Soc_2(A) \leq C_p \oplus R_p = P + R_p$ for every $p\in \pi(T)$. Therefore, $\Soc_2(A) \ast \Soc_2(A) \leq \bigcap_{p \in \pi(T)} P + R_p$. 

Set $U:= \bigcap_{p\in \pi(T)} P + R_p$ and denote by $U_p$ the $p$-component of $U$, $p \in \pi(T)$. Fix $p \in \pi(T)$. Clearly, $C_p \leq U_p$, as $P \leq U$. On the other hand, we also have that $U_p \leq U \leq P + R_p = C_p + R_p$. Thus, $U_p = C_p + (U_p \cap R_p)$. Since $R_p$ does not contain $p$-elements, $U \cap R_p = \{0\}$, and therefore, $U_p = C_p$. Hence, $U = P$ and we have
\[ \Soc_2(A) \ast \Soc_2(A) \leq \bigcap_{p \in \pi(T)} P + R_p = P. \qedhere \]
\end{proof}

\begin{corolari}
\label{cor:2.15nou}
Let $A$ be a Dedekind left brace such that $A = \Soc_2(A)$ and the additive group of $A$ is not periodic. Then, the following statements hold:
\begin{itemize}
\item $\zeta(A) = \Soc(A)$ and the additive group of $A/\zeta(A)$ is periodic.
\item $\zeta(A)$ includes an abelian ideal $P$ isomorphic to a periodic and locally cyclic and $A/P$ is an abelian brace.
\end{itemize}
\end{corolari}

\begin{proof}
It follows from Propositions~\ref{prop:2.14nou} and~\ref{prop:2.4nou}, and Corollary~\ref{cor:2.13nou}.
\end{proof}

\begin{proof}[Proof of Theorem~\ref{ateo:2.17nou}]
By Corollary~\ref{cor:2.15nou}, $\zeta(A)$ includes an abelian ideal $P$ isomorphic to a periodic and locally cyclic group, and such that $A \ast A \leq P$. Choose in $\zeta(A)$ a subbrace $B$ which is maximal under the property $B \cap P  = \{0\}$. Suppose that the additive group of the quotient $\zeta(A)/B$ is not periodic, i.e. there exists $u +B \in \zeta(A)/B$ with infinite additive order. Set $U/B:= \langle u + B \rangle_+$, where $U$ is a subbrace of $\zeta(A)$. Since the additive group of $(P + B)/B \cong P$ is periodic, the intersection $U/B \cap (P + B)/B$ is zero. It follows that $U \cap P = \{0\}$, and we obtain a contradiction. This contradiction shows that the additive group of $\zeta(A)/B$ is periodic. Moreover, the additive group of $A/\zeta(A)$ is also periodic by Corollary~\ref{cor:2.15nou}. Therefore, the additive group of $T := A/B$ is periodic. 

Now, since $P \leq \zeta(A)$, we have that $(P + B)/B \leq \zeta(A/B)$. On the other hand, since $A \ast A\leq  P$ and $T \ast T = (A/B) \ast (A/B) = (A \ast A + B)/B$, it follows that $T \ast T \leq (P+B)/B$. Therefore, $T \ast T \leq \zeta(T)$ and the additive group of $T \ast T$ is periodic and locally cyclic. 

Finally, take $D := A/P$. Since $B \cap P = \{0\}$, it follows that $A$ is isomorphic to some subbrace of the direct sum $A/B \oplus A/P$, which proves the theorem.
\end{proof}

\section*{Acknowledgements}
The first, second and fourth authors are supported by the Ministerio de Ciencia, Innovación y Universidades, Agencia Estatal de Investigación, FEDER (grant: PID2024-159495NB-I00), and the Conselleria d'Educació, Universitats i Ocupació, Generalitat Valenciana (grant: \mbox{CIAICO/2023/007}). The third author is very grateful to the Conselleria d'Innovaci\'o, Universitats, Ci\`encia i
Societat Digital of the Generalitat (Valencian Community, Spain) and the Universitat de
Val\`encia for their financial support to host researchers affected by the war in
Ukraine in research centres of the Valencian Community.

\end{document}